\documentclass[12pt]{amsart}
\usepackage{amsfonts}
\usepackage{txfonts}
\usepackage[round,authoryear]{natbib}
\usepackage{graphicx}
\usepackage{float}
\usepackage{amssymb,amsmath,amsthm,mathrsfs}
\numberwithin{equation}{section}

\newtheorem{thm}{Theorem}[section]
\newtheorem{lemma}[thm]{Lemma}
\newtheorem{remark}[thm]{Remark}

\newtheorem{definition}[thm]{Definition}

\newcommand{\bs}{{\bf S}}
\newcommand{\bi}{{\bf I}}
\newcommand{\bX}{{\bf X}}
\newcommand{\bY}{{\bf Y}}
\newcommand{\bt}{{\bf T}}
\newcommand{\bx}{{\bf x}}
\newcommand{\by}{{\bf y}}
\newcommand{\bd}{{\bf D}}
\newcommand{\bb}{{\bf B}}
\newcommand{\ba}{{\bf A}}
\newcommand{\br}{{\bf r}}
\newcommand{\bR}{{\bf R}}
\newcommand{\bS}{\boldsymbol\Sigma}
\renewcommand{\(}{\left(}
\renewcommand{\)}{\right)}

\newcommand{\E}{{\rm E}}
\newcommand{\tr}{{\rm tr}}
\newcommand{\mb}{\mathbf}

\newcommand{\re}{{\rm E}}
\newcommand{\rtr}{{\rm tr}}

\allowdisplaybreaks

\begin{document}

\title[singular value distribution of data matrices]{On the singular value distribution of large-dimensional data matrices whose columns have different correlations}
\author{Yanqing Yin}
\thanks{Yanqing Yin was partially supported by an NSF Grant
China 11701234, the Priority Academic Program Development of Jiangsu Higher Education Institutions
and a Research Support Project of Jiangsu Normal University (17XLR014)).}
\address{School of Mathematics and Statistics, Jiangsu Normal University, Xuzhou, P.R.C., 221116.}
\email{yinyq@jsnu.edu.cn}

\begin{abstract}
Suppose $\bY_n=\(\mb y_1,\cdots,\mb y_n\)$ is a $p\times n$ data matrix whose columns $\mb y_j, 1\leq j\leq n$ have different correlations. The asymptotic spectral property of $\bs_n=\frac1n\bY_n\bY^*_n$ when $p$ increases with $n$ has recently been considered by some authors. This model has become increasingly popular because of its wide applications in multi-user multiple-input single-output (MISO) systems and robust signal processing. In this paper, for more convenient applications in practice, we will investigate the spectral distribution of $\bs_n$ under milder moment conditions than the existing work. Some applications of this model are also discussed.
\end{abstract}


\maketitle

\section{Introduction and motivation}
The spectral analysis of sample covariance matrices has drawn increasing attention in recent years. Many methods of statistical inference involving population covariance matrices require the investigation of the properties, particularly the spectral property of the sample covariance matrices. See, for instance, \cite{anderson1958introduction,Johnstone2008Multivariate,Johnstone2009APPROXIMATE,Johnstone2009On,Paul2014Random}. Consider the $p\times n$ data matrix $\mb D_n=\(\mb d_1,\cdots,\mb d_n\)$, where $\mb d_j, 1\leq j\leq n$ are $n$ independent samples drawn from a $p$-dimensional distribution with zero mean and covariance matrix $\Sigma_p$. Then, when $p$ is fixed and $n\to \infty$, the sample covariance matrix $\frac1n\mb D_n\mb D_n^*$ is a consistent estimator of $\Sigma_p$. However, statistics has opened a new area, where we must work with more complex data. This shift challenges the classical theory in statistics and spurs the developments of new theories. Recent advances in random matrix theory (RMT) have clearly shown that in the asymptotic regime, where $p$ and $n$ go to infinity at the same pace, the sample covariance matrix is no longer consistent. To illustrate this phenomenon, we first introduce the following definitions.
\begin{definition}
Let $\ba$ be an $n \times n$ Hermitian matrix, and denote its eigenvalues by ${\lambda_j}(\ba), j = 1,2, \cdots, n$. Then, the empirical spectral distribution (ESD) of $\ba$ is defined by
$$F^{\ba}\left(x\right) =\frac{1}{n}\sum\limits_{j = 1}^n {I\left({\lambda _j}(\ba) \le x\right)},$$ where ${I\left(\mathbb{A}\right)}$ is the indicator function of an event ${\mathbb{A}}$.
\end{definition}
\begin{definition}
The Stieltjes transform of ${F^{\mathbf A}}\left(x\right)$, which is the ESD of an $n \times n$ Hermitian  matrix $\mb A$, is given by
$${ m}_{F^{\ba}}\left(z\right)=\int_{-\infty}^{+\infty}\frac{1}{x-z}d{F^{\mathbf A}}\left(x\right),$$
where $z=u+ iv\in\mathbb{C}^+\cup {\rm Supp}(F^{\mathbf A})^c$. Here ${\rm Supp}(F^{\mathbf A})^c$ stands for the complement of the support of $F^{\mathbf A}$.
\end{definition}
Suppose $\mb d_j=\mb \Sigma_p^{1/2}\mb x_j, 1\leq j\leq n$, where $\Sigma_p^{1/2}$ is the Hermitian square root of $\Sigma_p$, and $\mb X_n=\(\mb x_1,\mb x_2,\cdots,\mb x_n\)$ is a $p\times n$ matrix, whose elements are i.i.d. complex random variables with 0 means and unit variances. When $n\to \infty$, $p/n\to c\in(0,\infty),$ $F^{\bS_p}\xrightarrow{d}{H}$, and the sequence $\(\mb \Sigma_p\)_n$ is bounded in the spectral norm. Then, the famous M-P law, which was first proven in \cite{marchenko1967distribution} and developed by \cite{Wachter1978The,RePEc:eee:jmvana:v:20:y:1986:i:1:p:50-68,silverstein1995empirical,silverstein1995strong}, states that almost surely, the ESD $F^{\bs_n}$ of the sample covariance matrix $\frac1n\mb D_n\mb D_n^*=\frac{1}{n}\mb \Sigma_p^{1/2}\mb X_n\mb X_n^*\mb \Sigma_p^{1/2}$ weakly tends to a nonrandom p.d.f. $F$ as $n\to \infty$. For each $z\in\mathbb{C}^+$, $m(z)=m_F(z)$ is the unique solution to the equation
\begin{align}\label{al3}
m(z)=\int\frac1{t(1-c-czm(z))-z}dH(t)
\end{align}
in the set $\left\{m(z)\in\mathbb{C}^+: -(1-c)/z+cm(z)\in\mathbb{C}^+\right\}$.
For more details, we refer the reader to \cite{marchenko1967distribution,bai2010spectral,anderson2010introduction,pastur2011eigenvalue}.

As previously mentioned, statisticians are currently facing increasingly complex data. Thus, much effort has been devoted to developing the theoretical results and making them more applicable in practice. \cite{Bai2008Large} worked with matrices of the form $\frac1n\mb D_n\mb D_n^*$, where the columns of $\mb D_n$ are independent and share the same covariance matrices. They showed the validity of the M-P law when a condition of quadratic forms is satisfied. Another significant development in this direction was made by \cite{Zhang2006Spectral,Karoui2009Concentration}. Under some assumptions, they proved the existence of the limiting ESD for separable sample covariance matrices of the form $\frac1n \mb T_{2,n}^{1/2}\mb X_n\mb T_{1,n}\mb X_n^*\mb T_{2,n}^{1/2}$. The Stieltjes transform of the limiting ESD is the unique solution of a coupled functional equation.
In fact, they show that under the conditions of
\begin{itemize}
\item[(1)] $\mb X_n=(x_{jl})$ is $N\times n$, consisting of independent standard complex random variables satisfying the Lindeberg-type condition, i.e., for each $\delta>0$, as $n\to\infty$, $
    \frac1{\delta^2nN}\sum_{j,l}\re\left(|x_{jl}|^2I\left(|x_{jl}|>\delta\sqrt n\right)\right)\to0;$
\item[(2)] $\bt_{1n}$ is an $n\times n$ Hermitian matrix and $\bt_{2n}$ is an $N\times N$ nonnegative definite Hermitian matrix, both of which are independent of $\mb X_n$;
\item[(3)] with probability 1, as $n\to\infty$, the empirical spectral distributions of $\bt_{1n}$ and $\bt_{2n}$, which are denoted by $H_{1n}$ and $H_{2n}$, weakly converge to two probability functions $H_1$ and $H_2$, respectively; and
\item[(4)] $n/N\to c>0$ when $n\to\infty$,
\end{itemize}
with probability $1$, as $n\to\infty$, if $H_1$ and $H_2$ are not degenerate, then the ESD of $\bb_n=\frac1N\bt_{2n}^{1/2}\bX_n\bt_{1n}\bX_n^*\bt_{2n}^{1/2}$ weakly converges to a non-random probability distribution function $F$, whose Stieltjes transform $m(z)$ is determined by the following system of equations (\ref{eqs}), for each $z\in\mathbb{C}^+$,
\begin{align}\label{eqs}
\begin{cases}
m(z)=-z^{-1}(1-c)-z^{-1}c\int\frac1{1+q(z)x}dH_1(x)\\
m(z)=-z^{-1}\int\frac1{1+p(z)y}dH_2(y)\\
m(z)=-z^{-1}-p(z)q(z).
\end{cases}
\end{align}

In this paper, we consider data matrices whose columns may have different correlations. The matrix model, denoted as the different correlation model (DCM) for convenience, is defined as follows.
\begin{definition}[DMC]\label{def}
Assume that
\begin{itemize}
\item[(a)] $\{\mb x_{j,k}\},j,k=1,2,\cdots,$ are independent and identically distributed (i.i.d.) complex random variables with mean zero, variance 1 and $\re|x_{11}|^{4}=\mu<\infty $;
\item[(b)] $\bY_n=(y_{jk})$ is a $p\times n$  matrix and $\by_k=(y_{1k},\cdots,y_{pk})'=\bb_k\bx_k$ for $k=1,\cdots,n$;
\item[(c)] for $k=1,\cdots,n$, $\bb_k$ is a $p\times m_k$ non-random complex matrix, $\bS_k=\bb_k\bb_k^*$, and the spectral norm of $\bS_k$, which is denoted as $\|\bS_k\|$, is bounded in $n$;
\item[(d)] $\ba_n$ is a $p\times p$ Hermitian nonnegative definite matrix; and
\item[(e)] $c_n=p/n\to c\in(0,\infty)$, $c_{nk}=m_k/n\to c_k\in(0,\infty)$ as $n\to\infty$ and $\limsup_{n\to\infty}\sup_{1\le k\le n}c_k\le M$.
\end{itemize}
Then, $\bs_n=\frac1n\bY_n\bY^*_n+\ba_n$ is a matrix that follows the DCM.
\end{definition}
This model has attracted increasing popularity because of its wide applications in multi-user multiple-input single-output (MISO) systems and robust signal processing, as will be shown later.
\cite{Sebastian2012Large} first introduced the above model and showed some spectral property of $\bs_n$, but the entries in their model are assumed to have at least a finite eight-order moment, which is much higher than ours. Then, in \cite{Kammoun2016No}, the authors proved the no-outside results for this model when the random variables were Gaussian.

Our main result of this paper is as follows:
\begin{thm}\label{thd1}
Suppose $\bs_n$ is a matrix that follows the DCM; then, for any $z\in\mathbb{C}^+$, as $n\to\infty$, the distance between the Stieltjes transforms of $F^{\bs_n}(x)$ and \begin{align}\label{eqst0}
m_n^0(z)=\frac1p\rtr\left(\frac1n\sum_{k=1}^n\frac{\bS_k}{1+e_{nk}(z)}+\ba_n-z\bi_p\right)^{-1}
\end{align} convergence almost surely to $0$,
where the functions $e_{n1}(z), \cdots, e_{nn}(z)$ form the unique solution of
\begin{align}\label{eqst}
e_{nk}(z)=\frac1n\rtr\left[\bS_k\left(\frac1n\sum_{j=1}^n\frac{\bS_j}{1+e_{nj}(z)}+\ba_n-z\bi_p\right)^{-1}\right],
\end{align}
which is the Stieltjes transformation of a nonnegative finite measure on $\mathbb{R}^+$.
\end{thm}

\begin{remark}
We can easily verify that when all $\bS_k$s are equal and $\mb A_n=\mb 0$ for all $n$, our theorem consists of the well-known M-P law. In fact, suppose $\Sigma_k=\Sigma$ for $k=1,\cdots,n$ and $\mb A_n=\mb 0$ for all $n$ in (\ref{eqst}), we have all $e_{nk}(z), k=1,\cdots,n$  are equal for given $z$. Denote $e_{nk}(z)=e_n(z), k=1,\cdots,n$ and combine (\ref{eqst0}) and (\ref{eqst}) we have
\begin{align*}
e_{n}(z)&=\frac1n\rtr\left[\bS\left(\frac1n\sum_{j=1}^n\frac{\bS}{1+e_{n}(z)}-z\bi_p\right)^{-1}\right]
=\frac1n\rtr\left[\bS\left(\frac{\bS}{1+e_{n}(z)}-z\bi_p\right)^{-1}\right]\\\notag
&=\frac1n\rtr\left[\({1+e_{n}(z)}\)\(\frac{\bS}{1+e_{n}(z)}-z\bi_p+z\bi_p\)\left(\frac{\bS}{1+e_{n}(z)}-z\bi_p\right)^{-1}\right]\\\notag
&=c_n\(1+e_{n}(z)\)+c_nz\(1+e_{n}(z)\)\frac{1}{p}\rtr\left[\left(\frac{\bS}{1+e_{n}(z)}-z\bi_p\right)^{-1}\right]\\\notag
&=\(1+e_{n}(z)\)\(c_n+c_nzm_n^0(z)\),
\end{align*}
which yields
\begin{align}\label{eqst1}
\frac{1}{1+e_{n}(z)}=1-\frac{e_{n}(z)}{1+e_{n}(z)}=1-c_n+c_nzm_n^0(z).
\end{align}
Combining (\ref{eqst0}) and (\ref{eqst1}) we obtain
$$m_n^0(z)=\frac1p\rtr\left(\bS\(1-c_n+c_nzm_n^0(z)\)-z\bi_p\right)^{-1},$$
and thus (\ref{al3}) follows.
\end{remark}

\begin{remark}
One may be concerned with whether our main theorem remains valid when $m_k$ is arbitrary or infinite. In this case, our main theorem remains valid if the random variables $x_{j,k}$ have a finite six-order moment. This issue can be achieved by some modifications in the truncation step, as shown in the following section.
\end{remark}

The remainder of the paper is organized as follows. In Section 2, some applications of this model are  presented. Section 3 concern the proof of the main theorem, and Section 4 lists some necessary lemmas.

\section{Applications of the model}
In this section, we give some applications of the introduced model.
\subsection{Application in the Multiple-Input Single-Output Channel}
Consider the downlink of a single-cell system, in which a base station
with $n$ antennas serves $p$ users, each of whom is equipped with a single antenna, and assume that $n<p$. Then, the downlink channel vector $\mb z_k$ between the base station  and the k-th user is given by \cite{Sebastian2012Large}
$$\mb z_k=\mb \Phi_k\mb \varepsilon_k, k=1,\cdots,n,$$
where $\mb \varepsilon_k$ is a standard complex random vector, and $\mb \Phi_k$ describes the channel correlation of user $k$. The analysis of the spectrum of $\mb Z\mb Z^*$, where $\mb Z=\(\mb z_1,\cdots,\mb z_n\)$, is essential in the analysis of the MISO systems, and our main theorem can be applied in this case. Since the applications in this direction have been discussed in details in \cite{Sebastian2012Large}, we omit the repetitive discussion. For further details, one may refer to the original paper of \cite{Sebastian2012Large}.

\subsection{Application in a problem of sample classification}
Classification and Cluster analysis are two important problems in multivariate statistical analysis and machine learning. The former mainly concern in identifying which of a set of categories (sub-populations) a new observation belongs to, on the basis of the knowledge of the sub-populations or a training set of data containing observations whose category membership is known. While the task of the latter is grouping a set of sample in the manner of similar (in some sense). As two examples of the more general problem of pattern recognition, classification and   Cluster find applications in various aspects of modern science and thus attract many attentions for long times, see \cite{Collins2004Introduction} for instance.

In this subsection, we consider a usual problem, which can be seen as an example of sample classification. We then proposed a method as an initial solution to this problem by applying our main theorem of this paper.
\subsubsection{Statement of the problem and proposed method} Suppose we have $n$ samples, denoted as $\{\mb y_1,\mb y_2,\cdots, \mb y_n\}$, each of which drawn from exact one of several $p$-dimensional populations denoted as $\{G_1, G_2, \cdots, G_q\}$. The task is to distinguish the affiliation of each sample. That is to say, we need to determine which population that each sample is drawn from. That is a practical problem since we may lose the affiliation of mixed samples for some reasons.

Consider the simple case where $q=2$. In most cases, the exact probability distributions of $G_1$ and $G_2$ are hard to known. Denote $\mu_1, \mu_2$ and $\Sigma_1, \Sigma_2$ as the mean vectors and population covariance matrices of the two populations $G_1$ and $G_2$. Those parameter are assumed to be known. If we further assume that the difference between $\mu_1$ and $\mu_2$ is significant enough and the two populations are both gaussian, then we shall apply the classical Distance Discrimination method. However, few literature concern about the situation when the difference between two mean vectors is negligible while the covariance matrices are different from each other. In this subsection, as a direct application of our main theorem, we consider how to classify the samples into two categories when $\mu_1=\mu_2$, $\Sigma_1\neq \Sigma_2$.
This problem can be seen as a kind of classification in some sense but not the classical one.

Note that without lose of generality, we shall assume $\mu_1=\mu_2=\mb 0$ in what follows. We also meed the assumption that $\mb S_n=\frac{1}{n}\sum_{i=1}^n\mb y_i\mb y_i'$ follows the DCM. Denote $n_1$ and $n_2$ as the true number of samples come from $G_1$ and $G_2$ respectively. Our proposed method for solving this specified problem is as follows.
First, set $z=u_0+iv_0$ with $v_0$ large (for example set $v_0=100$). Then
\begin{description}
  \item[Step 1 \ Estimate $n_1$($n_2$)]Calculate $m_n(z)=\frac{1}{p}\tr(\mb S_n-z\mb I_p)^{-1}.$ For $k=0,\cdots,n$, solve the system of equations
\begin{align*}
\begin{cases}
 \ e_{1}(k,z)=\frac1n\rtr\left[\bS_1\left(\frac{k}{n}\frac{\bS_1}{1+e_{1}(k,z)}+\frac{n-k}{n}\frac{\bS_2}{1+e_{2}(k,z)}-z\bi_p\right)^{-1}\right],\\
 \ e_{2}(k,z)=\frac1n\rtr\left[\bS_2\left(\frac{k}{n}\frac{\bS_1}{1+e_{1}(k,z)}+\frac{n-k}{n}\frac{\bS_2}{1+e_{2}(k,z)} -z\bi_p\right)^{-1}\right],\\
\end{cases}
\end{align*}
then calculate $m_n^{0}(k,z)=\frac1p\rtr\left(\frac{k}{n}\frac{\bS_1}{1+e_{1}(k,z)}+\frac{n-k}{n}\frac{\bS_2}{1+e_{2}(k,z)}-z\bi_p\right)^{-1}.$
We estimate the number $n_1$ as $${\hat n_1}=\arg \min_{k=0,\cdots,n}|m_n^{0}(k,z)-m_n(z)|.$$
  \item[Step 2 \ Classify the samples] For $l=1,\cdots,n$, let $$\mb S_{n,l}=\frac{1}{n-1}\sum_{i\neq l}\mb y_i\mb y_i'.$$ Calculate $m_{n,l}(z)=\frac{1}{p}\tr(\mb S_{n,l}-z\mb I_p)^{-1}.$
 For $k_l=0,\cdots,n-1$, Solving the  system of equations
\begin{align*}
\begin{cases}
 \ e_{1}(k_l,z)=\frac1n\rtr\left[\bS_1\left(\frac{k_l}{n-1}\frac{\bS_1}{1+e_{1}(k_l,z)}+\frac{n-1-k_l}{n-1}\frac{\bS_2}{1+e_{2}(k_l,z)}-z\bi_p\right)^{-1}\right],\\
 \ e_{2}(k_l,z)=\frac1n\rtr\left[\bS_2\left(\frac{k_l}{n-1}\frac{\bS_1}{1+e_{1}(k_l,z)}+\frac{n-1-k_l}{n-1}\frac{\bS_2}{1+e_{2}(k_l,z)} -z\bi_p\right)^{-1}\right],\\
\end{cases}
\end{align*}
then calculate $m_{n,l}^{0}(k_l,z)=\frac1p\rtr\left(\frac{k_l}{n-1}\frac{\bS_1}{1+e_{1}(k_l,z)}+\frac{n-1-k_l}{n-1}\frac{\bS_2}{1+e_{2}(k_l,z)}-z\bi_p\right)^{-1}.$

Set $${\hat n_{1,l}}=\arg \min_{k_l=0,\cdots,n-1}|m_{n,l}^{0}(k_l,z)-m_{n,l}(z)|.$$
Then
\begin{align*}
{\mbox {the sample} \ \mb y_l \mbox { is  labeled  with}}
\begin{cases}
\ 1, \ & \mbox {if} \ \hat n_{1,l} < \hat n_{1}.\\
\ 2, \ & \mbox {if} \ \hat n_{1,l}\geq \hat n_{1}.\\
\end{cases}
\end{align*}
We determine that the samples labeled with 1 are drawn from $G_1$ while the samples labeled with 2 are drawn from $G_2$.
\end{description}

\subsubsection{Simulation studies}
In this section, we conduct some simulation results to investigate the finite sample performance of our proposed method for sample classification.

For given $n_1, n_2, p$, set $n=n_1+n_2$, $\mb \Sigma_1=2\mb I_p$ and $\mb \Sigma_2=(\sigma_{2,i,j}),$ where for $1\leq i,j\leq p$, $\sigma_{2,i,j}=0.5^{|i-j|}.$
Generate $\mb X=(x_{i,j})\triangleq\(\mb x_1,\cdots,\mb x_n\)$ as a $p\times n$ matrix with its entries $x_{i,j}\sim Gamma(4,0.5)$ are independent and identically distributed. Let $\mb 1_p$ be a $p$ dimensional vector with all its entries equal 1.  Let $\mb Y=\(\mb y_1,\cdots,\mb y_n\)\triangleq\(\mb Y_1,\mb Y_2\)$ where $\mb Y_1=\mb \Sigma_1^{1/2}\(\mb x_1-2\mb 1_p,\cdots,\mb x_{n_1}-2\mb 1_p\)$ and $\mb Y_2=\mb \Sigma_2^{1/2}\(\mb x_{n_1+1}-2\mb 1_p,\cdots,\mb x_{n}-2\mb 1_p\).$ Then the data matrices $\mb Y_1$ and $\mb Y_2$ are $n_1$ and $n_2$ samples drawn from two populations $G_1\sim (\mb 0, \mb \Sigma_1)$ and $G_2\sim (\mb 0, \mb \Sigma_2)$ respectively. We use the proposed method in this subsection to discriminate to which population each sample $\mb y_i$ $(1\leq i\leq n)$ belongs.
For given pair of $p,n$ we repeat the above procedure 100 times and draw the histogram of the numbers of wrongly classified samples as well as the empirical cumulative distribution function. The empirical results for several different pairs of $p,n$ are showed in Fig. \ref{fig1}-\ref{fig6}.

The simulations showed that our proposed method have good performances. For example, when $p=200, n=400, n1=n1=200,$ there are 95 times of classification procedure over the whole 100 times result in a number of wrongly classified samples smaller than 40.

\begin{figure}[H]
  \includegraphics[width=0.7\textwidth]{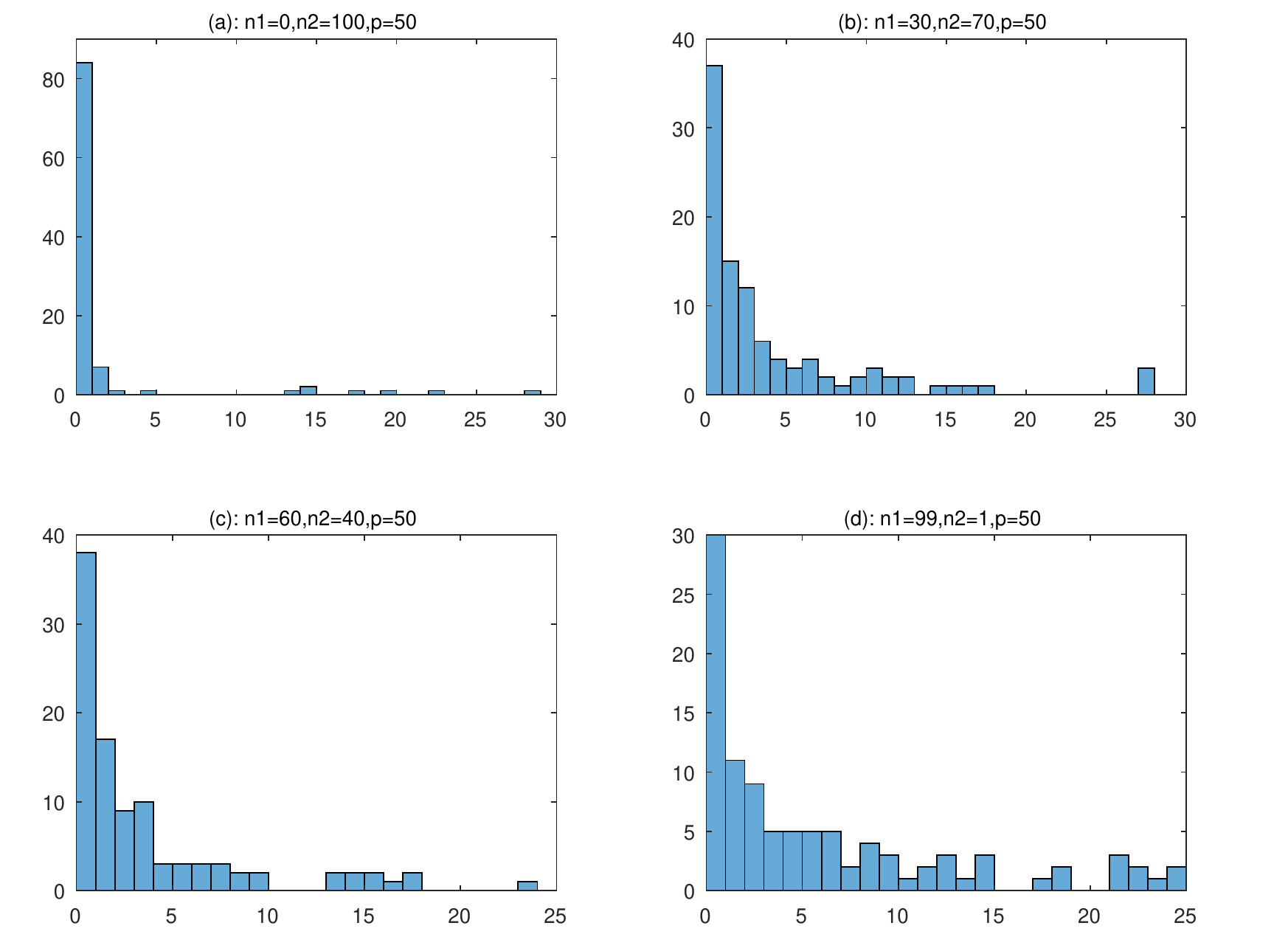}
  \caption{Histogram of the numbers of wrongly classified samples when $p=50,n=100$.}\label{fig1}	
\end{figure}

\begin{figure}[H]
  \includegraphics[width=0.7\textwidth]{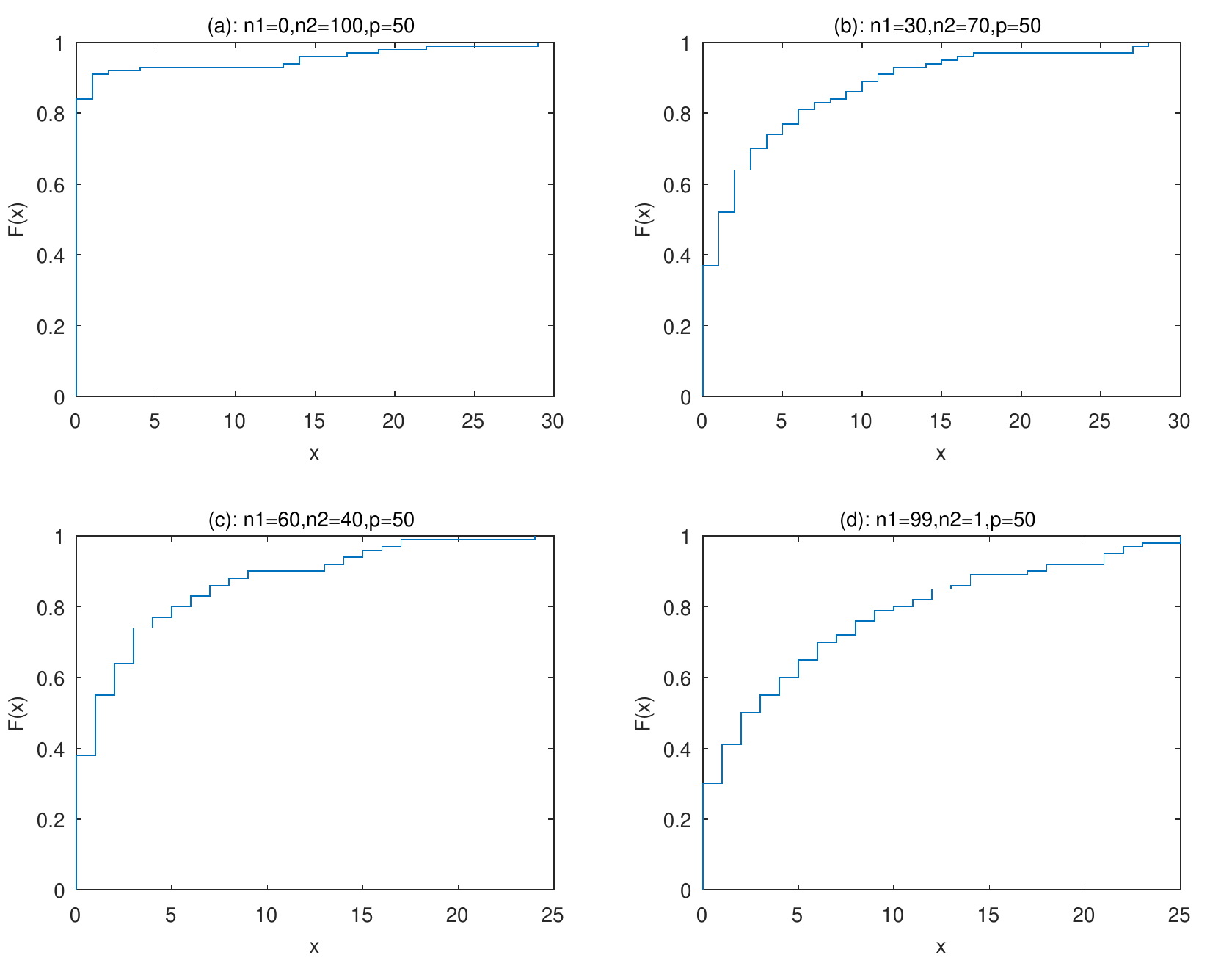}
  \caption{Empirical cumulative distribution function of the numbers of wrongly classified samples when $p=50,n=100$.}\label{fig2}	
\end{figure}

\begin{figure}[H]
  \includegraphics[width=0.7\textwidth]{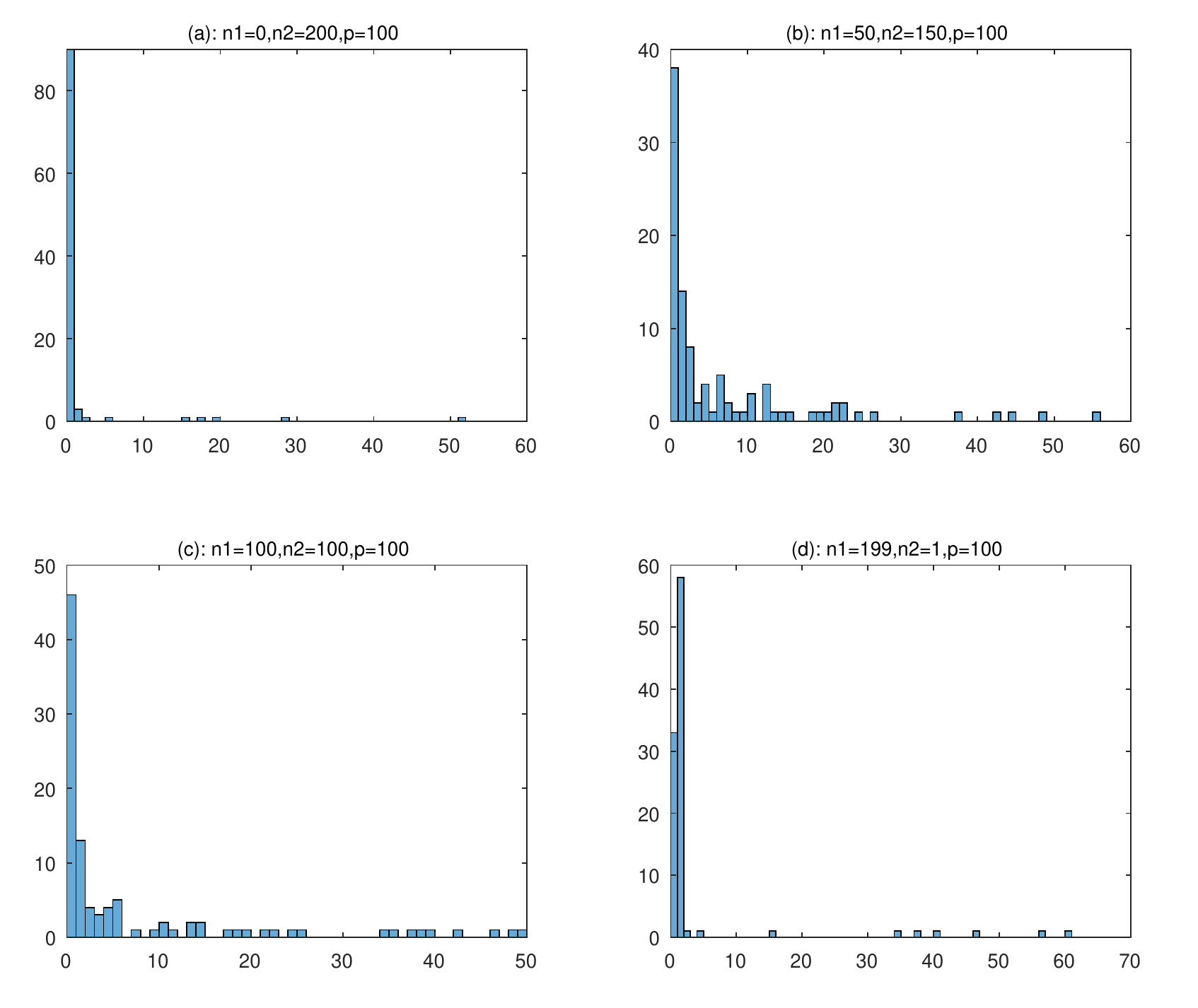}
  \caption{Histogram of the numbers of wrongly classified samples when $p=100,n=200$.}\label{fig3}	
\end{figure}

\begin{figure}[H]
  \includegraphics[width=0.7\textwidth]{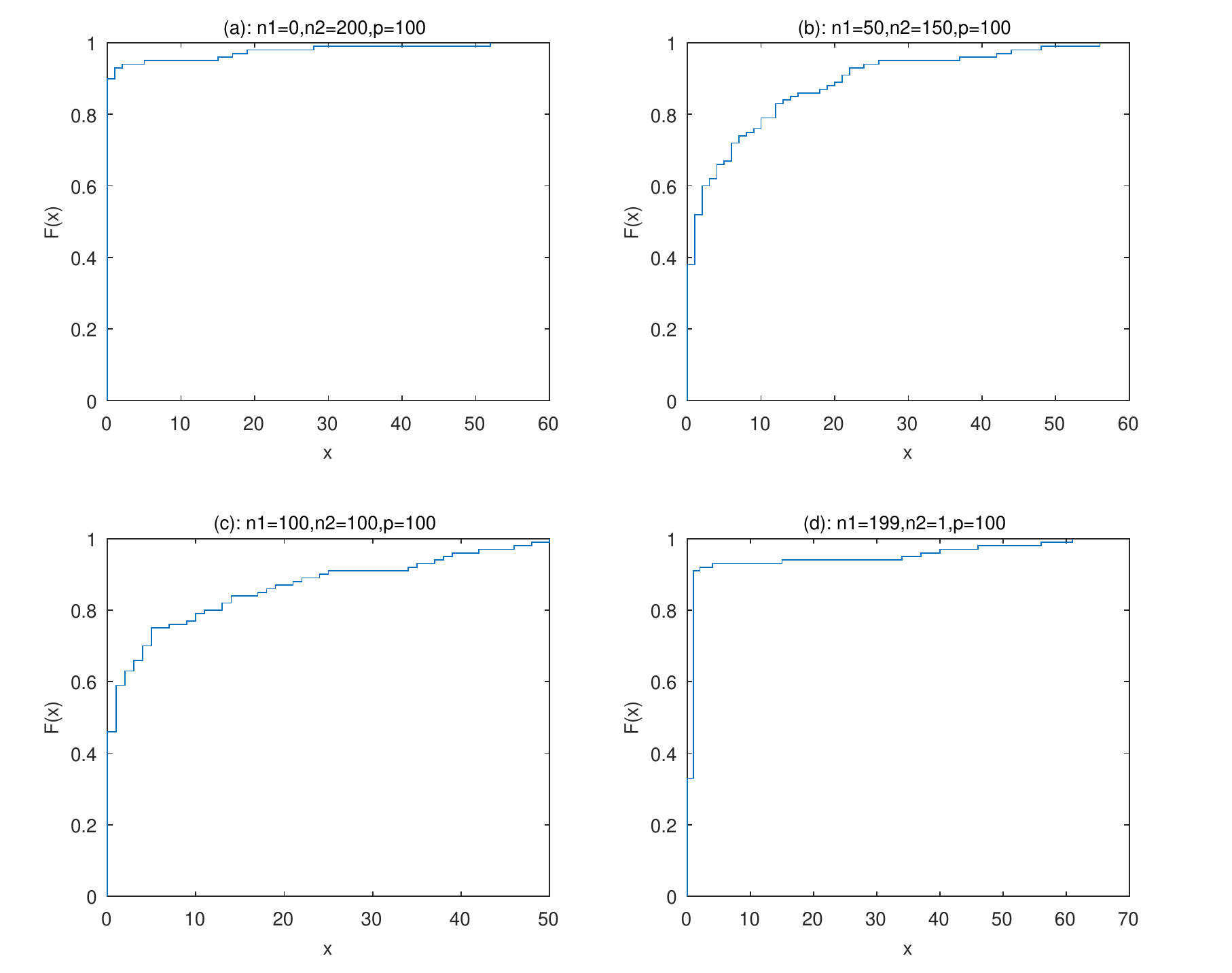}
  \caption{Empirical cumulative distribution function of the numbers of wrongly classified samples when $p=100,n=200$.}\label{fig4}	
\end{figure}

\begin{figure}[H]
  \includegraphics[width=0.7\textwidth]{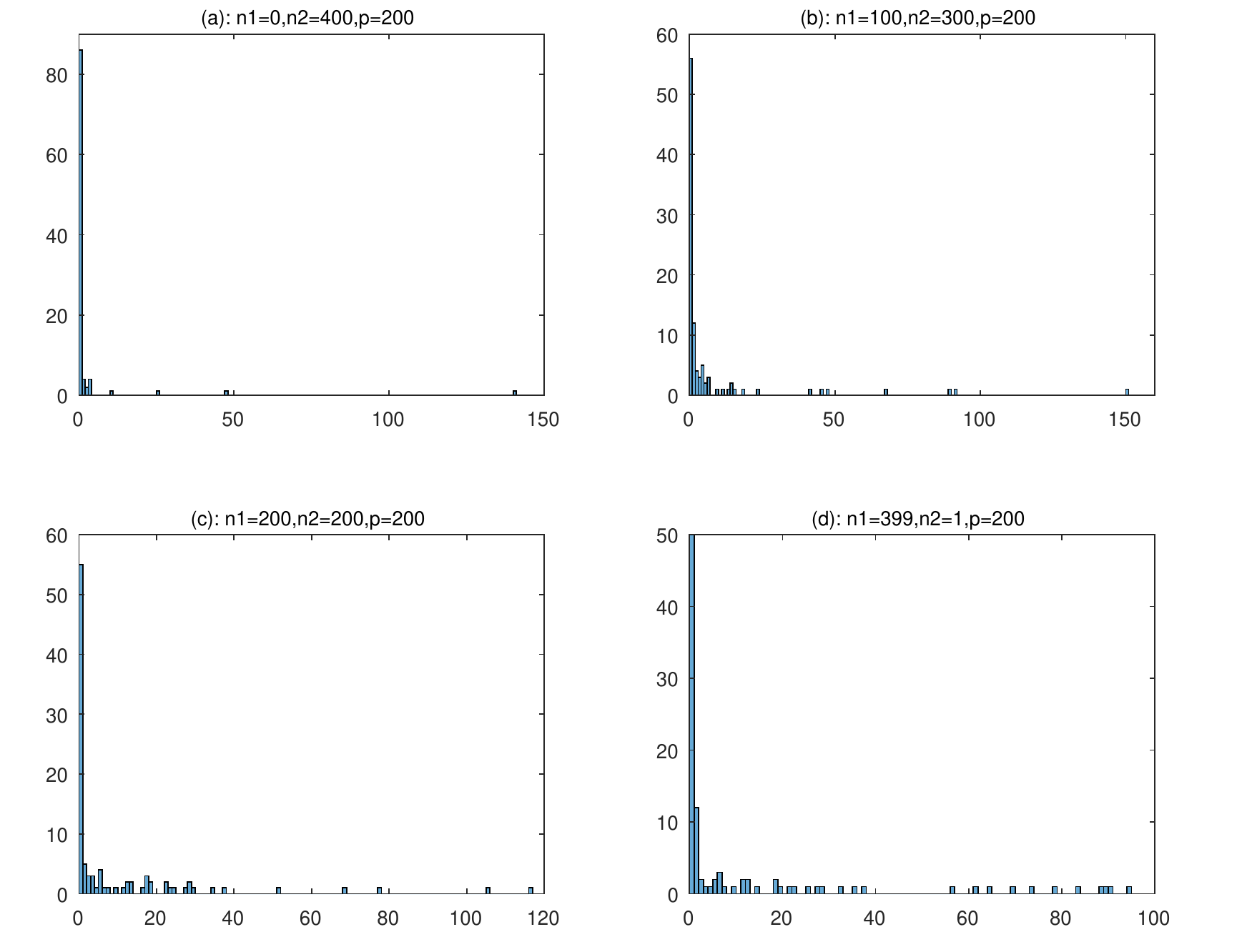}
  \caption{Histogram of the numbers of wrongly classified samples when $p=200,n=400$.}\label{fig5}	
\end{figure}

\begin{figure}[H]
  \includegraphics[width=0.7\textwidth]{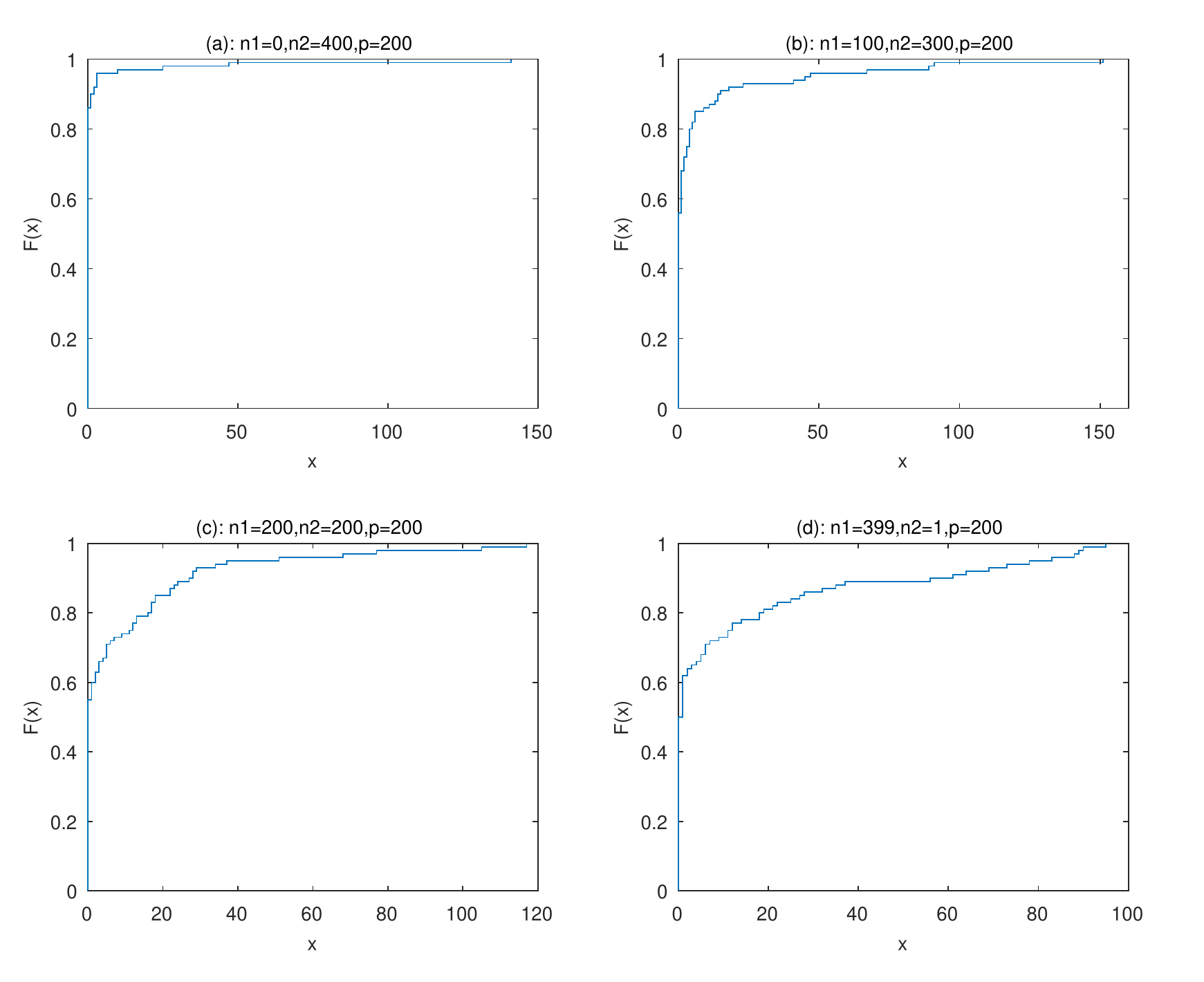}
  \caption{Empirical cumulative distribution function of the numbers of wrongly classified samples when $p=200,n=400$.}\label{fig6}	
\end{figure}

\begin{remark}
For large $n$ and $p$, the calculation cost of solving the system of equations (\ref{eqst}) could be huge. However, the calculation cost can be reduced a lot by applying the following Lemma which was proved in \cite{Sebastian2012Large}.
\begin{lemma}
Let $z=u+iv$ with $v>0$ and $\{e_{nl}^{(k)}(z)\}(k\geq 0)$ be a sequence defined by $e_{nl}^{(0)}(z)=-\frac{1}{z}$ and
 $$e_{nl}^{(k)}(z)=\frac1n\rtr\left[\bS_l\left(\frac1n\sum_{j=1}^n\frac{\bS_j}{1+e_{nl}^{(k-1)}(z)}+\ba_n-z\bi_p\right)^{-1}\right]$$
 for $k>0$. Then for any $l=1,\cdots,n$, $\lim_{k\to\infty}e_{nl}^{(k)}(z)=e_{nl}(z)$ defined in (\ref{eqst}).
\end{lemma}
\end{remark}



\section{Proof of the main theorem}

This section shows the proof of our main theorem. To relax the condition on the moment of random variables, we must truncate the variables at a proper order by obtaining a bound on the spectral norm of $\bs_n$. In turn, the truncation is achieved by applying the non-asymptotic analysis of random matrices. Henceforth, $C$ denotes a constant that may take different values from one appearance to another.

\subsection{A bound on the spectral norm of $\bs_n$}
This part aims to give a   bound on the largest eigenvalue of $\bs_n$. We need the following lemma, which is a modification of Theorem 5.44 in \cite{ver2010}.
\begin{lemma}\label{vsn}
Let $\bf A$ be an $N\times M$ matrix, whose rows $A_i$ are independent zero-mean random row vectors in $\mathbb{C}^M$ with covariance matrices $\mb \Sigma_i=\rm E A_i^*A_i$. Let $l$ be a number such that $\sqrt{A_iA_i^*}\leq \sqrt{l}$ and $\|\mb \Sigma_i\|\leq C\leq l$ almost surely for all $i$. Then, for every $t>0$, the following inequality holds with a probability of at least $1-2M\exp^{-ct^2/2}$:
\begin{align}
  \|{\bf A}\|\leq {C}\sqrt{N}+ t \sqrt{l},
\end{align}
where $c$ is a constant.
\end{lemma}

The proof of Lemma \ref{vsn} depends on the following lemma in \cite{Joel2012User}.
\begin{lemma}[Non-commutative Bernstein-type inequality]\label{bti}
Consider a finite-sequence ${\mb \Psi_i}$ of independent centered self-adjoint random $M\times M$ matrices. Assume that for some numbers $K$ and $\sigma$, we have
$$\|\mb \Psi_i\|\leq k \ {\rm almost \ surely},\quad \|\sum_i\E \mb \Psi_i^2\|\leq \sigma^2.$$
Then, for every $t\geq 0$. we have
$${\rm P}\(\|\sum_i\ \mb \Psi_i\|\geq t\)\leq 2M\exp\(\frac{-t^2/2}{\sigma^2+Kt/3}\).$$
\end{lemma}
\begin{proof}[Proof of Lemma \ref{vsn}]
First, write $$\frac{1}{N}\mb A^*\mb A-\frac{1}{N}\sum_{i=1}^N\mb \Sigma_i =\frac{1}{N}\sum_{i=1}^N \(A_i^*A_i-\mb \Sigma_i\)\triangleq \sum_{i=1}^N\mb \Psi_i.$$
It is easy to verify that for any $i$,
$$\E\(\mb \Psi_i\)=\E\frac{1}{N}\(A_i^*A_i-\mb \Sigma_i\)=\mb 0,$$
and
$$\|\mb \Psi_i\|\leq \frac{1}{N}\(\|A_i^*A_i\|+\|\mb \Sigma_i\|\)\leq \frac{l+\|\mb \Sigma_i\|}{N}\leq \frac{2l}{N}.$$
Here, we use $\|A_i^*A_i\|={A_iA_i^*}\leq l$.
In addition, we have
$$\mb \Psi_i^2=\frac{1}{N^2}\(\(A_i^*A_i\)^2-2\(A_i^*A_i\)\mb \Sigma_i+\mb \Sigma_i^2\),$$
which implies
$$\E\mb \Psi_i^2=\frac{1}{N^2}\(\E\(A_i^*A_i\)^2-\mb \Sigma_i^2\).$$
Thus, we have $$\|\E\mb \Psi_i^2\|\leq \frac{1}{N^2}\|\mb \Sigma_i\|\(l+\|\mb \Sigma_i\|\)$$
by noting that $\(A_i^*A_i\)^2=A_iA_i^*\(A_i^*A_i\).$
Then, we arrive at $$\|\sum_{i=1}^N\E\mb \Psi_i^2\|\leq \frac{1}{N^2}\sum_{i=1}^N\|\mb \Sigma_i\|\(l+\|\mb \Sigma_i\|\)=\frac{1}{N}\sum_{i=1}^N\|\mb \Sigma_i\|\(\frac{1}{N}\(l+\|\mb \Sigma_i\|\)\)\leq \frac{2Cl}{N}.$$
Now, denoting $\zeta=\max\(t\sqrt{l/N},t^2l/N\)$ and applying Lemma \ref{bti}, we obtain
\begin{align}
&{\rm P}\(\|\frac{1}{N}\mb A^*\mb A-\frac{1}{N}\sum_{i=1}^N\mb \Sigma_i\|\geq \zeta\)\\\notag
&\leq 2M\exp\(-c\min\(\frac{\zeta}{\frac{2l}{N}},\frac{\zeta^2}{\frac{2l}{N}}\)\)
=2M\exp\(-ct^2\frac{l}{N}{\frac{N }{2l}}\)=2M\exp\(-\frac{c t^2}{2}\).
\end{align}
Here, we use the simple fact that $\min\(\(\max\(x,x^2\)\),\(\max\(x,x^2\)\)^2\)=x^2$.
This completes the proof of this lemma.
\end{proof}

Let $A_i=\frac{1}{\sqrt{n}}\mb x_i^*\mb B_i^*$, $N=n$, and $M=p$. We have $\|\mb \Sigma_i\|^{1/2}=\|\rm E A_iA_i'\|^{1/2}\leq C$. Then, applying Lemma \ref{led2}, for any $i$, we have
\begin{align*}
{\rm P}({|\mb x_i^*{\bf B_n^*}{\bf B_n}\mb x_i-\rtr{\bf B_n}{\bf B_n^*}|>n^{5/3}})\leq \frac{{\rm E}|\mb x_i^*{\bf B_n^*}{\bf B_n}\mb x_i-\rtr{\bf B_n}{\bf B_n^*}|^2}{n^{10/3}}\leq \frac{Cn}{n^{10/3}}=\frac{C}{n^{7/3}}.
\end{align*}
Thus, we obtain $${\rm P}(\sup_{i}{A_iA_i'>Cn^{5/3}})\leq \frac{Cn}{n^{7/3}}=\frac{C}{n^{4/3}},$$ which is summable.

Then, letting $t=C\log n$ and $l=Cn^{5/3}$, based on Lemma \ref{vsn}, for any $s$ and all large $n$, we arrive at
\begin{align}\label{al11}
{\rm P}(\|{\bf S_n}\|>n^{6/7})=o(n^{-s}).
\end{align}

This implies
\begin{align}\label{max}
  \|\bs_n\|\le n^{6/7}\quad a.s.
\end{align}

\subsection{Truncation and recentralization}
In this subsection, we truncate the variables in the data matrix at a proper position. The truncation step is an important tool in relaxing moment conditions in various random matrix models. The use of this technique can be found in many papers, see for instance, \cite{Yin1988} and \cite{Karoui2009}.

Suppose that the assumptions of Theorem \ref{thd1} hold. Since $\re|x_{11}|^{4}<\infty$, for any $\tau>0$, we have
\begin{align*}
\sum_{k=1}^{\infty}\tau^{-2}2^{2k}{\rm P}\left(|x_{11}|\ge\tau2^{k/2}\right)<\infty.
\end{align*}
Then, we can select a slowly decreasing sequence of constants $\tau_n\to0$ such that
\begin{align*}
\sum_{k=1}^{\infty}\tau_{2^k}^{-2}2^{2k}{\rm P}\left(|x_{11}|\ge\tau_{2^k}2^{k/2}\right)<\infty.
\end{align*}
We truncate the variables $x_{jk}$ at $\tau_n\sqrt n$ and denote the resulting variables by $\hat x_{jk}$, i.e., $\hat {x}_{jk}=x_{jk}I\left(|x_{jk}|\leq \tau_n\sqrt n\right)$. We also denote
$$\hat\bx_k=\left(\hat {x}_{1k},\cdots,\hat x_{m_kk}\right)'~~\mbox{and}~~ \ \widehat\bs_n=\frac{1}{n}\sum_{k=1}^n\bb_k\hat\bx_k\hat\bx_k^*\bb_k^*$$
and obtain
\begin{align*}
&P(\bs_n\neq \widehat\bs_n,i.o.)=\lim_{N\to\infty}{\rm P}\left(\bigcup_{n=2^N}^{\infty}\bigcup_{k=1}^n\{\bx_k\neq \hat\bx_k\}\right)\\
\le&\lim_{N\to\infty}\sum_{l=N}^{\infty}{\rm P}\left(\bigcup_{2^l<n\le2^{l+1}}\bigcup_{k=1}^n\bigcup_{j=1}^{m_k}\{x_{jk}\neq \hat x_{jk}\}\right)\\
\le&\lim_{N\to\infty}\sum_{l=N}^{\infty}{\rm P}\left(\bigcup_{2^l<n\le2^{l+1}}\bigcup_{k=1}^{2^{l+1}}\bigcup_{j=1}^{2M2^{l+1}}\{|x_{jk}|\ge\tau_{2^{l}}2^{l/2}\}\right)\\
=&\lim_{N\to\infty}\sum_{l=N}^{\infty}{\rm P}\left(\bigcup_{k=1}^{2^{l+1}}\bigcup_{j=1}^{2M2^{l+1}}\{|x_{jk}|\ge\tau_{2^l}2^{l/2}\}\right)\\
\le&8M\lim_{N\to\infty}\sum_{l=N}^{\infty}2^{2l}{\rm P}\left(|x_{11}|\ge\tau_{2^l}2^{l/2}\right)\to0.
\end{align*}

Let $\widetilde\bs_n=\frac{1}{n}\sum_{k=1}^n\bb_k\left(\hat\bx_k-\re\hat\bx_k\right)\left(\hat\bx_k-\re\hat\bx_k\right)^*\bb_k^*$. It is obvious that
\begin{align*}
|\re{\hat x}_{jk}|\le\re\left|x_{jk}I(|x_{jk}|>\tau_n\sqrt n)\right|\le\frac{\mu}{\tau_n^3n^{3/2}}.
\end{align*}
Using Lemma \ref{led1}, we obtain
\begin{align*}
&\max_j|\lambda_j(\widehat\bs_n)-\lambda_j(\widetilde\bs_n)|\le\|\widehat\bs_n-\widetilde\bs_n\|\\
\le&2\left|\frac1n\sum_{k=1}^n(\re\hat\bx_k)^*\bb_k^*\bb_k\hat\bx_k\right|+\left|\frac1n\sum_{k=1}^n(\re\hat\bx_k)^*\bb_k^*\bb_k(\re\hat\bx_k)\right|\\
\le&\frac{2C}n\sum_{k=1}^n|(\re\hat\bx_k)^*\hat\bx_k|+\frac Cn\sum_{k=1}^n|(\re\hat\bx_k)^*(\re\hat\bx_k)|\\
\le&\frac{2C}n\sum_{k=1}^n\sum_{j=1}^{m_k}|\hat x_{jk}\re\overline{\hat x}_{jk}|+\frac Cn\sum_{k=1}^n\sum_{j=1}^{m_k}|\re{\hat x}_{jk}|^2\\
\le&\frac{2C}n\sum_{k=1}^n\sum_{j=1}^{m_k}|\hat x_{jk}\re\overline{\hat x}_{jk}-|\re{\hat x}_{jk}|^2|+\frac {2C}n\sum_{k=1}^n\sum_{j=1}^{m_k}|\re{\hat x}_{jk}|^2\\
\le&\frac{2C}n\sum_{k=1}^n\sum_{j=1}^{m_k}|\hat x_{jk}\re\overline{\hat x}_{jk}-|\re{\hat x}_{jk}|^2|+\frac {CM\mu^2}{\tau_n^6n^{2}}\xrightarrow{a.s.}0.
\end{align*}
Here, we use the fact that for any $\varepsilon>0$,
\begin{align*}
&{\rm P}\left(n|\hat x_{jk}\re\overline{\hat x}_{jk}-|\re{\hat x}_{jk}|^2|>\varepsilon\right)\le\frac{2n^4\re|\hat x_{jk}\re\overline{\hat x}_{jk}|^4}{\varepsilon^4}\le\frac{2n^4\mu|\re{\hat x}_{jk}|^4}{\varepsilon^4}\le\frac{2\mu}{\varepsilon^4\tau_n^{12}n^2},
\end{align*}
which is summable.

Let $\sigma^2=\re|\hat x_{11}-\re\hat x_{11}|^2$; then, we have
\begin{align*}
 1-\sigma^2=&\E \left|x_{11}\right|^2\( I\left(|x_{11}|>\tau_n\sqrt n\right)+I\left(|x_{11}|\leq \tau_n\sqrt n\right)\)-\re|\hat x_{11}|^2+\(\re\hat x_{11}\)^2\\\notag
=&\re\left|x_{11}\right|^2 I\left(|x_{11}|>\tau_n\sqrt n\right)+\(\re\hat x_{11}\)^2\\\notag
\le&2\re\left|x_{11}\right|^2 I\left(|x_{11}|>\tau_n\sqrt n\right)\le\frac{2\mu}{\tau_n^2n}.
\end{align*}
From Lemma \ref{led1} and (\ref{max}), this yields
\begin{align*}
\max_j|\lambda_j(\widetilde\bs_n)-\lambda_j(\sigma^{-2}\widetilde\bs_n)|\le\frac{1-\sigma^2}{\sigma^2}\|\widetilde\bs_n\|
\le\frac{2 C\mu}{\tau_n^2n}\|\widetilde\bs_n\|\xrightarrow{a.s.}0.
\end{align*}

For simplicity, the truncated and recentralized variables are still denoted by $x_{jk}$. We assume the following:
\begin{itemize}
\item[(1)] The variables $\{x_{jk},j=1,2,\cdots,m;k=1,2,\cdots,n\}$ are independent.
\item[(2)] $\re(x_{jk})=0$ and ${\rm Var}(x_{jk})=1$.
\item[(3)] $|x_{jk}|\le \tau_n\sqrt n$.
\item[(4)] $\re|x_{11}|^{4}\le \mu$.
\end{itemize}
Then, we will prove Theorem \ref{thd1} under the above conditions.

\subsection{The proof of Theorem \ref{thd1}}
We begin by providing some necessary definitions and primary results that will be used in the proof.

Let $m_n(z)=m_{F^{\bs_n}}(z)$ and $\bd_n(z)=\bs_n-z\bi_p$. Define $\bd_{nk}(z)=\bd_n(z)-\frac1n\bb_k\bx_k\bx_k^*\bb_k^*,$
\begin{align*}
&\bR_n(z)=\frac1n\sum_{k=1}^n\beta_{nk}(z)\bS_k+\ba_n-z\bi_p,\quad \widehat\bR_{nj}(z)=\frac1n\sum_{k=1}^n\hat\beta_{nkj}(z)\bS_k+\ba_n-z\bi_p,\\
&\bR(z)=\frac1n\sum_{k=1}^n\frac1{1+e_{nk}(z)}\bS_k+\ba_n-z\bi_p,
\end{align*}
and
\begin{align*}
&\rho_{nk}(z)=\frac1{1+n^{-1}\bx_k^*\bb_k^*\bd_{nk}^{-1}(z)\bb_k\bx_k},\quad \theta_{nk}(z)=\frac1{1+n^{-1}\rtr\left[\bS_k\bR_{n}^{-1}(z)\right]},\\
&\beta_{nk}(z)=\frac1{1+n^{-1}\rtr\left[\bS_k\bd_{n}^{-1}(z)\right]},\quad \hat\beta_{nkj}(z)=\frac1{1+n^{-1}\rtr\left[\bS_k\bd_{nj}^{-1}(z)\right]}.
\end{align*}

It can be verified that
\begin{align}\label{ald1}
\left\|\bR_n^{-1}(z)\right\|\le\frac1{\Im z}\quad{\rm for } \ z\in \mathbb{C}^+.
\end{align}
In fact, we have
\begin{align*}
\bR_n(z)=&\left\{\frac1n\sum_{k=1}^n\frac{1+n^{-1}\Re\rtr\left[\bS_k\bd_{n}^{-1}(\bar z)\right]}{\left|1+n^{-1}\rtr\left[\bS_k\bd_{n}^{-1}(z)\right]\right|^2}\bS_k+\ba_n-\Re z\bi_p\right\}\\
&-i{\Im z}\left\{\frac1{n^2}\sum_{k=1}^n\frac{\rtr\left[\bS_k\bd_{n}^{-1}(z)\bd_{n}^{-1}(\bar z)\right]}{\left|1+n^{-1}\rtr\left[\bS_k\bd_{n}^{-1}(\bar z)\right]\right|^2}\bS_k+\bi_p\right\}\\
\triangleq&\bR_{n1}(z)-i\bR_{n2}(z),
\end{align*}
where both $\bR_{n1}(z)$ and $\bR_{n2}(z)$ are Hermitian matrices. Let ${\bf u}_j$ be the unit eigenvector of $\bR_n(z)$ that corresponds to $\lambda_j\left(\bR_n(z)\right)$. Then, we obtain
\begin{align*}
\left|\lambda_j\left(\bR_n(z)\right)\right|=&\left|{\bf u}_j^*\bR_{n1}(z){\bf u}_j-i{\bf u}_j^*\bR_{n2}(z){\bf u}_j\right|\\
\ge&\left|{\bf u}_j^*\bR_{n2}(z){\bf u}_j\right|\ge\lambda_{\min}\left(\bR_{n2}(z)\right)\ge{\Im z}.
\end{align*}
This finishes the proof of (\ref{ald1}). Using the same argument, it follows that
\begin{align}\label{ald2}
\left\|\widehat\bR_{nj}^{-1}(z)\right\|\le\frac1{\Im z}\quad{\rm for } \ z\in \mathbb{C}^+.
\end{align}

Now, we can show the proof of our main theorem. We shall proceed with three steps:
\subsubsection{Convergence of $m_n(z)-\frac1p\rtr\bR_n^{-1}(z)$}
Write
\begin{align*}
\bd_n(z)-\bR_n(z)=\frac1n\sum_{k=1}^n\bb_k\bx_k\bx_k^*\bb_k^*-\frac1n\sum_{k=1}^n\beta_{nk}(z)\bS_k.
\end{align*}
Taking the inverses and using the well-known formula
\begin{align}\label{ald4}
\left(\ba+\br\br^*\right)^{-1}=\ba^{-1}-\frac{\ba^{-1}\br\br^*\ba^{-1}}{1+\br^*\ba^{-1}\br},
\end{align}
we obtain
\begin{align*}
&\bd_n^{-1}(z)-\bR_n^{-1}(z)=-\frac1n\sum_{k=1}^n\bR_n^{-1}(z)\left[\bb_k\bx_k\bx_k^*\bb_k^*\bd_n^{-1}(z)-\beta_{nk}(z)\bS_k\bd_n^{-1}(z)\right]\\
=&-\frac1n\sum_{k=1}^n\rho_{nk}(z)\bR_n^{-1}(z)\bb_k\bx_k\bx_k^*\bb_k^*\bd_{nk}^{-1}(z)+\frac1n\sum_{k=1}^n\beta_{nk}(z)\bR_n^{-1}(z)\bS_k\bd_n^{-1}(z).
\end{align*}
For any $p\times p$ Hermitian matrix $\bt_n$ with a bounded spectral norm, it follows that
\begin{align*}
\bt_n\bd_n^{-1}(z)-\bt_n\bR_n^{-1}(z)=&-\frac1n\sum_{k=1}^n\rho_{nk}(z)\bt_n\bR_n^{-1}(z)\bb_k\bx_k\bx_k^*\bb_k^*\bd_{nk}^{-1}(z)\\
&+\frac1n\sum_{k=1}^n\beta_{nk}(z)\bt_n\bR_n^{-1}(z)\bS_k\bd_n^{-1}(z).
\end{align*}
Taking the trace and dividing by $p$, one finds
\begin{align*}
w_{\bt_n}(z)\triangleq&\frac1p\rtr\left[\bt_n\bd_n^{-1}(z)\right]-\frac1p\rtr\left[\bt_n\bR_n^{-1}(z)\right]
=\frac1{pn}\sum_{k=1}^n\beta_{nk}(z)\rtr\left[\bt_n\bR_n^{-1}(z)\bS_k\bd_n^{-1}(z)\right]\\
&-\frac1{pn}\sum_{k=1}^n\rho_{nk}(z)\bx_k^*\bb_k^*\bd_{nk}^{-1}(z)\bt_n\bR_n^{-1}(z)\bb_k\bx_k\\
=&-\frac1{pn}\sum_{k=1}^n\rho_{nk}(z)\bx_k^*\bb_k^*\bd_{nk}^{-1}(z)\bt_n\left[\bR_n^{-1}(z)-\widehat\bR_{nk}^{-1}(z)\right]\bb_k\bx_k\\
&-\frac1{pn}\sum_{k=1}^n\rho_{nk}(z)\left[\bx_k^*\bb_k^*\bd_{nk}^{-1}(z)\bt_n\widehat\bR_{nk}^{-1}(z)\bb_k\bx_k
-\rtr\left(\bd_{nk}^{-1}(z)\bt_n\widehat\bR_{nk}^{-1}(z)\bS_k\right)\right]\\
&-\frac1{pn}\sum_{k=1}^n\rho_{nk}(z)\left[\rtr\left(\bd_{nk}^{-1}(z)\bt_n\widehat\bR_{nk}^{-1}(z)\bS_k\right)
-\rtr\left(\bd_{nk}^{-1}(z)\bt_n\bR_n^{-1}(z)\bS_k\right)\right]\\
&-\frac1{pn}\sum_{k=1}^n\rho_{nk}(z)\left[\rtr\left(\bd_{nk}^{-1}(z)\bt_n\bR_n^{-1}(z)\bS_k\right)
-\rtr\left(\bd_{n}^{-1}(z)\bt_n\bR_n^{-1}(z)\bS_k\right)\right]\\
&-\frac1{pn}\sum_{k=1}^n\left[\rho_{nk}(z)-\beta_{nk}(z)\right]\rtr\left(\bd_{n}^{-1}(z)\bt_n\bR_n^{-1}(z)\bS_k\right)\\
\triangleq&\frac1{p}\sum_{k=1}^n\left(d_{k1}(z)+d_{k2}(z)+d_{k3}(z)+d_{k4}(z)+d_{k5}(z)\right).
\end{align*}
Following the same strategy that used in the proof of (3.4) in \cite{Bai1998No} or in the proof of (6.2.5) in \cite{bai2010spectral}, we can easily check that $\rho_{nk}(z)$, $\beta_{nk}(z)$, and $\hat\beta_{nkj}(z)$ are all bounded in absolute values by $|z|/\Im z$. Note that by (\ref{ald4}),
\begin{align}\label{ald3}
&\bR_n^{-1}(z)-\widehat\bR_{nk}^{-1}(z)=\frac1n\sum_{j=1}^n\left(\hat\beta_{njk}(z)-\beta_{nj}(z)\right)\bR_n^{-1}(z)\bS_j\widehat\bR_{nk}^{-1}(z)\\
=&-\frac1{n^3}\sum_{j=1}^n\hat\beta_{njk}(z)\beta_{nj}(z)\rho_{nk}(z)\left(\bx_k^*\bb_k^*\bd_{nk}^{-1}(z)\bS_j\bd_{nk}^{-1}(z)\bb_k\bx_k\right)\bR_n^{-1}(z)\bS_j\widehat\bR_{nk}^{-1}(z).\notag
\end{align}
Using (\ref{ald1}), (\ref{ald2}), (\ref{ald3}) and the fact that $\|(\ba-z\bi_n)^{-1}\|\le\frac1{\Im z}$ for any $n\times n$ Hermitian matrix $\ba$, we have
\begin{align*}
|d_{k1}(z)|\le&\frac{C|z|}{n(\Im z)^2}\left|\bx_k^*\bx_k\right|\left\|\bR_n^{-1}(z)-\widehat\bR_{nk}^{-1}(z)\right\|\\
\le&\frac{C|z|}{n^4(\Im z)^4}\left|\bx_k^*\bx_k\right|\sum_{j=1}^n\left|\hat\beta_{njk}(z)\beta_{nj}(z)\rho_{nk}(z)\bx_k^*\bb_k^*\bd_{nk}^{-1}(z)\bS_j\bd_{nk}^{-1}(z)\bb_k\bx_k\right|\\
\le&\frac{C|z|}{n^3(\Im z)^5}\left|\bx_k^*\bx_k\right|\sum_{j=1}^n\left|\hat\beta_{njk}(z)\beta_{nj}(z)\right|
\le\frac{C|z|^3}{n^2(\Im z)^7}\left|\bx_k^*\bx_k\right|.
\end{align*}
Here, we use the fact that
$$|n^{-1}\rho_{nk}(z)\(\bx_k^*\bb_k^*\bd_{nk}^{-1}(z)\bd_{nk}^{-1}(\bar {z})\bb_k\bx_k\)|\leq \frac{1}{\Im z}.$$
We note here that the above inequality will be used several times in the remainder of the paper.

By Lemma \ref{led2}, for $t\ge1$, we obtain
\begin{align*}
&\re|d_{k1}(z)|^{2t}\le\frac{C_t|z|^{6t}}{n^{4t}(\Im z)^{14t}}\re\left|\bx_k^*\bx_k\right|^{2t}\le\frac{C_t|z|^{6t}}{n^{4t}(\Im z)^{14t}}\left[\re\left|\bx_k^*\bx_k-\rtr(\bi_{m_k})\right|^{2t}+m_k^{2t}\right]\\
\le&\frac{C_t|z|^{6t}}{n^{4t}(\Im z)^{14t}}\left[m_k^{t}+\tau_n^{4t-4}n^{2t-2}m_k+m_k^{2t}\right]\le\frac{C_t|z|^{6t}}{n^{2 t}(\Im z)^{14t}},\notag
\end{align*}
which is summable.  Thus, as a consequence of Borel-Cantelli lemma, we arrive at
\begin{align}\label{ald5}
d_{k1}(z)\xrightarrow{a.s.}0.
\end{align}

From Lemma \ref{led2} and (\ref{ald2}), for any $t\ge1$, we obtain
\begin{align*}
\re|d_{k2}(z)|^{2t}\le&\frac{|z|^{2t}}{n^{2t}(\Im z)^{2t}}\re\left|\bx_k^*\bb_k^*\bd_{nk}^{-1}(z)\bt_n\widehat\bR_{nk}^{-1}(z)\bb_k\bx_k
-\rtr\left(\bd_{nk}^{-1}(z)\bt_n\widehat\bR_{nk}^{-1}(z)\bS_k\right)\right|^{2t}\\
\le&\frac{C_t|z|^{2t}}{n^{2t}(\Im z)^{2t}}\Bigg[\re\left(\rtr\left(\bd_{nk}^{-1}(z)\bt_n\widehat\bR_{nk}^{-1}(z)\bS_k\widehat\bR_{nk}^{-1}(\bar z)\bt_n\bd_{nk}^{-1}(\bar z)\bS_k\right)\right)^{t}\notag\\
&\quad\quad\quad+\tau_n^{4t-4}n^{2t-2}\re\rtr\left(\bd_{nk}^{-1}(z)\bt_n\widehat\bR_{nk}^{-1}(z)\bS_k\widehat\bR_{nk}^{-1}(\bar z)\bt_n\bd_{nk}^{-1}(\bar z)\bS_k\right)^{t}\Bigg]\notag\\
\le&\frac{C_t|z|^{2t}}{n^{2t}(\Im z)^{2t}}\left[\left(\frac n{(\Im z)^4}\right)^{t}+\tau_n^{4t-4}n^{2t-1}\frac1{(\Im z)^{4t}}\right]\le\frac{C_t|z|^{2t}}{(\Im z)^{6t}}\left(\frac1{n^t}+\frac{\tau_n^{4t-4}}n\right).\notag
\end{align*}
The last bound is summable when $t>1$, so by Borel-Cantelli lemma we have
\begin{align}\label{ald6}
d_{k2}(z)\xrightarrow{a.s.}0.
\end{align}

Based on (\ref{ald1}), (\ref{ald2}), and (\ref{ald3}), we have
\begin{align*}
|d_{k3}(z)|\le&\frac{C|z|}{n\Im z}\left|\rtr\left[\bS_k\bd_{nk}^{-1}(z)\bt_n\left(\widehat\bR_{nk}^{-1}(z)-\bR_n^{-1}(z)\right)\right]\right|\\
\le&\frac{C|z|}{n^3(\Im z)^2}\sum_{j=1}^n\left|\hat\beta_{njk}(z)\beta_{nj}(z)\rtr\left(\bS_k\bd_{nk}^{-1}(z)\bt_n\bR_n^{-1}(z)\bS_j\widehat\bR_{nk}^{-1}(z)\right)\right|\le\frac{C|z|^3}{n(\Im z)^7}.\notag
\end{align*}
Hence, by applying Borel-Cantelli lemma, we get
\begin{align}\label{ald7}
d_{k3}(z)\xrightarrow{a.s.}0.
\end{align}

Using (\ref{ald1}) and (\ref{ald4}), one finds that
\begin{align*}
|d_{k4}(z)|=&\left|\frac{\rho_{nk}^2(z)}{n^2}\bx_k^*\bb_k^*\bd_{nk}^{-1}(z)\bt_n\bR_n^{-1}(z)\bS_k\bd_{nk}^{-1}(z)\bb_k\bx_k\right|\le\frac{C|z|}{n(\Im z)^3}.
\end{align*}
Thus, we obtain
\begin{align}\label{ald8}
d_{k4}(z)\xrightarrow{a.s.}0.
\end{align}

Note that
\begin{align*}
\left|\rho_{nk}(z)-\beta_{nk}(z)\right|\le&\left|\rho_{nk}(z)-\hat\beta_{nkk}(z)\right|+\left|\hat\beta_{nkk}(z)-\beta_{nk}(z)\right|\\
=&\frac{|\rho_{nk}(z)\hat\beta_{nkk}(z)|}{n}\left|\bx_k^*\bb_k^*\bd_{nk}^{-1}(z)\bb_k\bx_k-\rtr\left(\bd_{nk}^{-1}(z)\bS_k\right)\right|\\
&+\frac{|\hat\beta_{nkk}(z)\beta_{nk}(z)\rho_{nk}(z)|}{n^2}\left|\bx_k^*\bb_k^*\bd_{nk}^{-1}(z)\bS_k\bd_{nk}^{-1}(z)\bb_k\bx_k\right|\\
\le&\frac{|z|^2}{n(\Im z)^2}\left|\bx_k^*\bb_k^*\bd_{nk}^{-1}(z)\bb_k\bx_k-\rtr\left(\bd_{nk}^{-1}(z)\bS_k\right)\right|+\frac{C|z|^2}{n(\Im z)^3}.
\end{align*}
Then, from Lemma \ref{led2}, for any $t\ge1$, we have
\begin{align*}
\re|d_{k5}(z)|^{2t}\le&\frac{C_t|z|^{4t}}{n^{2t}(\Im z)^{8t}}\re\left|\bx_k^*\bb_k^*\bd_{nk}^{-1}(z)\bb_k\bx_k-\rtr\left(\bd_{nk}^{-1}(z)\bS_k\right)\right|^{2t}+\frac{C_t|z|^{4t}}{n^{2t}(\Im z)^{10t}}\\
\le&\frac{C_t|z|^{4t}}{n^{2t}(\Im z)^{8t}}\left[\left(\frac n{(\Im z)^2}\right)^t+\tau_n^{4t-4}n^{2t-1}\frac1{(\Im z)^{2t}}\right]+\frac{C_t|z|^{4t}}{n^{2t}(\Im z)^{10t}}\notag\\
\le&\frac{C_t|z|^{4t}\tau_n^{4t-4}}{n(\Im z)^{10t}}+\frac{C_t|z|^{4t}}{n^{t}(\Im z)^{10t}}.\notag
\end{align*}
The last bound is summable when $t>1$, thus by Borel-Cantelli lemma   we arrive at
\begin{align}\label{ald9}
d_{k5}(z)\xrightarrow{a.s.}0.
\end{align}

Therefore, from (\ref{ald5})-(\ref{ald9}), we conclude that
\begin{align}\label{ald18}
w_{\bt_n}=\frac1{p}\sum_{k=1}^n\left(d_{k1}(z)+d_{k2}(z)+d_{k3}(z)+d_{k4}(z)+d_{k5}(z)\right)\xrightarrow{a.s.}0,
\end{align}
which implies that for fixed $z\in\mathbb{C}^+$,
\begin{align}\label{ald10}
w_{\bi_p}=m_n(z)-\frac1p\rtr\left(\bR_n^{-1}(z)\right)\xrightarrow{a.s.}0.
\end{align}

\subsubsection{Convergence of $\frac1n\rtr\left[\bS_k\bd_{n}^{-1}(z)\right]-e_{nk}(z)$}

Rewrite
\begin{align}\label{ald17}
s_{nk}\triangleq\frac1n\rtr\Big[\bS_k\bd_{n}^{-1}(z)\Big]-&e_{nk}(z)=c_nw_{\bS_k}+\frac1n\rtr\left[\bS_k\bR_{n}^{-1}(z)\right]-\frac1n\rtr\left[\bS_k\bR^{-1}(z)\right]\\
=&c_nw_{\bS_k}+\frac1n\rtr\left[\bS_k\left(\bR_{n}^{-1}(z)-\bR^{-1}(z)\right)\right]\notag\\
=&c_nw_{\bS_k}+\frac1{n^2}\sum_{j=1}^n\frac{\beta_{nj}(z)s_{nj}}{1+e_{nj}(z)}\rtr\left[\bS_k\bR_{n}^{-1}(z)\bS_j\bR^{-1}(z)\right].\notag
\end{align}
\cite{Sebastian2012Large} showed that $e_{nk}(z),k=1,\cdots,n$ are all Stieltjes transforms of the nonnegative finite measures on $\mathbb{R}^+$. Hence, by the same argument above inequality (12) in \cite{Paul2009No}, we have
\begin{align*}
\left|\frac1{1+e_{nk}(z)}\right|\le\frac{|z|}{\Im z}
\end{align*}
and
\begin{align*}
\left\|\bR^{-1}(z)\right\|\le\frac1{\Im z}.
\end{align*}
By (\ref{ald1}), it follows that
\begin{align*}
\left|s_{nk}\right|\le&c_n\left|w_{\bS_k}\right|
+\kappa\max_{1\le j\le n}\left|s_{nj}\right|,
\end{align*}
where $\max_{1\le j\le n}\|\bS_j\|\le\widehat M$, and $\kappa=\frac{(c+1)|z|^2\widehat M^2}{(\Im z)^4}$, which implies
\begin{align*}
\left(1-\kappa\right)\max_{1\le k\le n}\left|s_{nk}\right|\le&(c+1)\max_{1\le k\le n}\left|w_{\bS_k}\right|.
\end{align*}
On the set $\left\{z\in\mathbb{C}^+:0<\kappa<1\right\}$, one obtains for $\varepsilon>0$ and $t>1$
\begin{align*}
&{\rm P}\left\{\left(1-\kappa\right)\max_{1\le k\le n}\left|s_{nk}\right|>\varepsilon\right\}\le{\rm P}\left\{(c+1)\max_{1\le k\le n}\left|w_{\bS_k}\right|>\varepsilon\right\}\le\frac{(c+1)^{2t}}{\varepsilon^{2t}}\sum_{k=1}^n\re\left|w_{\bS_k}\right|^{2t},
\end{align*}
which is summable according to the last section. Consequently, we obtain for fixed $z\in\big\{z\in\mathbb{C}^+:0<\kappa<1\big\}$
\begin{align*}
\max_{1\le k\le n}\left|s_{nk}\right|\xrightarrow{a.s.}0.
\end{align*}
By Vitali's convergence theorem (Lemma \ref{led3}), we find that for all $z\in\mathbb{C}^+$,
\begin{align}\label{ald11}
\max_{1\le k\le n}\left|s_{nk}\right|\xrightarrow{a.s.}0.
\end{align}

\subsubsection{Completion of the proof of Theorem \ref{thd1}}

Note that
\begin{align*}
m_n(z)-m_n^0(z)=&w_{\bi_p}+\frac1p\rtr\left(\bR_n^{-1}(z)\right)-\frac1p\rtr\left(\bR^{-1}(z)\right)\\
=&w_{\bi_p}+\frac1{pn}\sum_{k=1}^n\frac{\beta_{nk}(z)s_{nk}}{1+e_{nk}(z)}\rtr\left(\bR_n^{-1}(z)\bS_k\bR^{-1}(z)\right).
\end{align*}
By (\ref{ald1}), (\ref{ald10}), and (\ref{ald11}), for fixed $z\in\mathbb{C}^+$, we have
\begin{align*}
\left|m_n(z)-m_n^0(z)\right|
\le&\left|w_{\bi_p}\right|+\frac{\widehat M|z|^2}{(\Im z)^4}\max_{1\le k\le n}\left|s_{nk}\right|\xrightarrow{a.s.}0.
\end{align*}
Using Vitali's convergence theorem, we have
\begin{align*}
m_n(z)-m_n^0(z)\xrightarrow{a.s.}0\quad{\rm for \ all} \ z\in\mathbb{C}^+.
\end{align*}

In \cite{Sebastian2012Large}, it has been shown that the functions $e_{n1}(z),\cdots,e_{nn}(z)$ form the unique solution of
\begin{align*}
e_{nk}(z)=\frac1n\rtr\left[\bS_k\bR^{-1}(z)\right],
\end{align*}
which is the Stieltjes transform of a nonnegative finite measure on $\mathbb{R}^+$. Thus, the proof of Theorem \ref{thd1} is complete.

\section{lemmas}

\begin{lemma}[Theorem A.46 of \cite{bai2010spectral}]\label{led1}
If $\ba$ and $\bb$ are Hermitian, then
\begin{align*}
\max_k|\lambda_k(\ba)-\lambda_k(\bb)|\le\|\ba-\bb\|.
\end{align*}
\end{lemma}

\begin{lemma}[Lemma B.26 of \cite{bai2010spectral}]\label{led2}
Let $\ba$ be an $n\times n$ nonrandom matrix and $\bx=(x_1,\cdots,x_n)'$ be a random vector of independent entries. Assume that $\re(x_j)=0$, $\re|x_j|^2$, and $\re|x_j|^l\le\nu_l$. Then, for any $t\ge1$,
\begin{align*}
\re\left|\bx^*\ba\bx-\rtr(\ba)\right|^t\le C_t\left[\left(\nu_4\rtr\left(\ba\ba^*\right)\right)^{t/2}+\nu_{2t}\rtr\left(\ba\ba^*\right)^{t/2}\right],
\end{align*}
where $C_t$ is a constant that only depends on $t$.
\end{lemma}

\begin{lemma}[Vitali's convergence theorem]\label{led3}
Let $f_1,f_2,\cdots$ be analytic in $D$, which is a connected open set of \ $\mathbb{C}$, satisfying $|f_n(z)|\le M$ for every $n$ and $z$ in $D$ and $f_n(z)$ converges as $n\to\infty$ for each $z$ in a subset of $D$ with a limit point in $D$. Then, there exists a function $f$ analytic in $D$ for which $f_n(z)\to f(z)$ for all $z\in D$.
\end{lemma}

\section{Acknowledgements}
The author wishes to thank the editor and two referees for their comments and valuable suggestions, which have led to great
improvements of this paper.


\end{document}